\documentclass[11pt]{amsart}
\usepackage[margin=1in]{geometry}
\usepackage{xcolor}
\numberwithin{equation}{section}
\newtheorem{thm}{Theorem}[section]

\newtheorem{lem}[thm]{Lemma}

\newtheorem{defn}[thm]{Definition}

\newtheorem{clm}[thm]{Claim}

\def\Xint#1{\mathchoice
	{\XXint\displaystyle\textstyle{#1}}%
	{\XXint\textstyle\scriptstyle{#1}}%
	{\XXint\scriptstyle\scriptscriptstyle{#1}}%
	{\XXint\scriptscriptstyle\scriptscriptstyle{#1}}%
	\!\int}
\def\XXint#1#2#3{{\setbox0=\hbox{$#1{#2#3}{\int}$ }
		\vcenter{\hbox{$#2#3$ }}\kern-.6\wd0}}

\def\dashint{\Xint-}

\newcommand{\divg}{\textup{div}}
\newcommand{\phx}{\Phi(x)}
\newcommand{\phkx}{\Phi_k(x)}
\newcommand{\ptl}{\tilde{p}}
\newcommand{\pxi}{\partial_{x_i}}
\newcommand{\bv}{\mathbf{v}}
\newcommand{\epc}{\frac{1}{\mu(u)}K(x)}
\newcommand{\ot}{\Omega_T}
\newcommand{\rn}{\mathbb{R}^N}
\newcommand{\qi}{q^I}
\newcommand{\qp}{q^P}
\newcommand{\phxt}{\tilde{\Phi}(\xi)}
\newcommand{\vt}{\tilde{v}}
\newcommand{\bvt}{\mathbf{\tilde{v}}}

\newcommand{\epck}{\frac{1}{\mu(u)}K_k(x)}
\newcommand{\itz}{\left(t_0-\frac{1}{2}r^2, t_0+\frac{1}{2}r^2\right)}
\newcommand{\pet}{\partial_\eta}
\newcommand{\qit}{\tilde{q^I}}

\newcommand{\ob}{\partial\Omega}
\newcommand{\pt}{\partial_t}
\newcommand{\wot}{W^{1,2}(\Omega)}
\newcommand{\br}{B_{r}(x_0)}

\newcommand{\bry}{B_{R}(y)}
\newcommand{\bryh}{B_{\frac{R}{2}}(y)}
\newcommand{\bdy}{B_{\delta}(y)}

\newcommand{\ibr}{\int_{B_{r}(x_0)}}
\newcommand{\ibry}{\int_{B_{R}(y)}}
\newcommand{\ibryh}{\int_{B_{\frac{R}{2}}(y)}}
\newcommand{\ibn}{\int_{B_{r_n}(0)}}
\newcommand{\ibnx}{\int_{B_{r_n}(x_0)}}

\newcommand{\aibr}{\dashint_{B_{r}(x_0)}}
\newcommand{\aibrh}{\dashint_{B_{\frac{r}{2}}(x_0)}}
\newcommand{\qr}{Q_{r}(z_0)}
\newcommand{\qrn}{Q_{r_n}(0)}
\newcommand{\qrnz}{Q_{r_n}(z_0)}
\newcommand{\qrh}{Q_{\frac{r}{2}}(z_0)}
\newcommand{\iqr}{\int_{Q_{r}(z_0)}}

\newcommand{\aiqr}{\dashint_{Q_{r}(z_0)}}

\newcommand{\ra}{\rightarrow}
\newcommand{\io}{\int_{\Omega}}
\newcommand{\iot}{\int_{\Omega_T}}
\newcommand{\ve}{\varepsilon}
\newcommand{\ek}{\eta_k}
\newcommand{\dk}{D_k(\bv)}
\newcommand{\vp}{\varphi}
\begin{document}
	\title[an elliptic-parabolic system modeling miscible fluid flows in porous media]{Partial regularity for an elliptic-parabolic system modeling miscible fluid flows in porous media}
	\author{Xiangsheng Xu}\thanks
	{Department of Mathematics and Statistics, Mississippi State
		University, Mississippi State, MS 39762.
		{\it Email}: xxu@math.msstate.edu.}
	\begin{abstract} In this paper we study the existence and partial regularity of weak solutions to an elliptic-parabolic system that models the single-phase miscible displacement of one incompressible fluid by another in a porous medium. The system is singular and involves quadratic nonlinearity. A partial regularity result is obtained via the Fefferman-Stein inequality.
\end{abstract}
\maketitle

\section{Introduction}
Let $\Omega$ be a bounded domain in $\rn$ with Lipschitz boundary $\ob$. For each $T>0$ we consider the initial boundary value problem
\begin{eqnarray}
	\divg\bv&=&\qi(x,t)-\qp(x,t)\ \ \mbox{in $\ot\equiv\Omega\times(0,T)$,}\label{pro1}\\
	\bv&=&-\epc\nabla p\ \ \mbox{in $\ot$,}\label{pro2}\\
	\phx\pt u-\divg\left(\phx D(\bv)\nabla u\right)+\nabla u\cdot\bv 
	&=&\qi(x,t)\left(\hat{u}(x,t)-u\right) \ \mbox{in $\ot$},\label{pro3}\\
	K(x)\nabla p\cdot\mathbf{n}&=&0\ \ \mbox{on $\Sigma_T\equiv\ob\times(0,T)$,}\label{pro4}\\
	D(\bv)\nabla u\cdot\mathbf{n}&=&0\ \ \mbox{on $\Sigma_T$,}\label{pro5}\\
	u(x,0)&=&u_0(x)\ \ \mbox{on $\Omega$},\label{pro6}
\end{eqnarray}
where $\mathbf{n}$ is the unit outward normal to $\ob$. This problem can be proposed as a mathematical model for the single-phase miscible displacement of one incompressible fluid by another in a porous medium \cite{AZ,B,CE,DO,F}. In this case, $p=p(x,t)$ is the pressure in the mixture, $\bv=\bv(x,t)$ is the Darcy velocity, and $u=u(x,t)$ is the concentration of the invading fluid. 
The remaining terms in the problem are given functions of their arguments. Their physical meanings and mathematical properties can be described as follows:
\begin{enumerate}
	\item[(H1)] The function $\phx$ is the porosity. It is assumed to satisfy
	$$\phx\in W^{1,2}(\Omega)\ \ \mbox{and}\ \ \phx \geq \lambda_0 \ \ \mbox{for some $\lambda_0>0$ and a.e. $x\in\Omega$}.$$
	\item[(H2)]The matrix $K(x)$ represents the absolute permeability of the porous medium. Entries of $K(x)$ are bounded measurable
	 functions on $\Omega$ and satisfy the uniform ellipticity condition, i.e, there is a positive number $c_0$ such that
	\begin{equation}
		\xi\cdot K(x)\xi\geq c_0|\xi|^2\ \ \ \mbox{for each $\xi\in\rn$ and a.e. $x\in\Omega$.}\nonumber
	\end{equation}
	\item[(H3)] The functions $\qi(x,t)$ and $\qp(x,t)$ are the injection and projection source terms, respectively, while $\hat{u}(x,t)$ is the injected concentration. They have the properties
	\begin{eqnarray}
		 0&\leq &\hat{u}\leq 1,\  \ \mbox{ $\qi\geq 0$, \  $\qp\geq 0$, }\nonumber\\
	\qi,\ \qp &\in &	 L^\infty(0,T; L^{\ell}(\Omega))\ \ \mbox{for each $\ell\geq 1$, and},\label{qcon}\\
	\io(\qi-\qp)dx&=&0\ \ \mbox{for a.e. $t\in (0,T)$}.\label{qcon2}
	\end{eqnarray}
Here and in what follows we suppress the independent variables of a function whenever there is no confusion. 
The last condition \eqref{qcon2} is imposed to ensure that \eqref{pro1}, \eqref{pro2}, and \eqref{pro4} are consistent.
\item[(H4)] The matrix $D(\bv)$ is given by
\begin{eqnarray}
	D(\bv)&=&mI+|\bv|\left(\frac{b\bv\otimes\bv}{|\bv|^2}+a\left(I-\frac{\bv\otimes\bv}{|\bv|^2}\right)\right)\nonumber\\
	&=&(m+a|\bv|)I+(b-a)\frac{\bv\otimes\bv}{|\bv|},\label{ddef}
\end{eqnarray}
where $\bv\otimes\bv=\bv\bv^T$, $I$ is the $N\times N$ identity matrix, and
$$m>0,\ \ b\geq a>0.$$
The product $\phx D(\bv)$ is the so-called hydrodynamic dispersion tensor. It is not difficult to check that $D(x)$ is a Lipschitz function of $x$ and satisfies
\begin{eqnarray}
	|D(\bv)|&\leq&b|\bv|+m,\label{ellip1}\nonumber\\
(a|\bv|+m)|\xi|^2&\leq &D(\bv)\xi\cdot\xi\leq (b|\bv|+m)|\xi|^2\ \ \mbox{for each $\xi\in\rn$.}\label{ellip}
\end{eqnarray}
Thus, equation \eqref{pro3} is singular on the set where $|\bv|$ is infinity. This constitutes the main difficulty in our mathematical analysis.
\item[(H5)] The function $\mu(u)$ is the viscosity, which is a continuous function of $u$ and satisfies
\begin{equation}
	\mu(u)\geq c_1\ \ \mbox{for some $c_1>0$.}\nonumber
\end{equation}
An example of $\mu(u)$ can be found in \cite{CE}.
\item[(H6)] $u_0(x)\in \wot\cap C^\alpha(\overline{\Omega})$ for some $\alpha\in (0,1)$. Furthermore, $0\leq u_0\leq 1$.
\end{enumerate}

As nonlinear problems need not have classical solutions, a suitable notion of a weak solution must be found for the problem. The following definition seems to be appropriate.
\begin{defn} Let (H1)-(H6) hold.
	We say $(p,u)$ is a weak solution to \eqref{pro1}-\eqref{pro6} if:
	\begin{enumerate}
		\item[(D1)] $u\in L^2(0,T;\wot)$ with $0\leq u\leq 1$, $p\in  L^\infty(0,T;\wot)$, and $\sqrt{1+|\bv|}|\nabla u|\in L^2(\ot)$, where $\bv$ is given as in \eqref{pro2};
		\item[(D2)]There hold
		\begin{equation}
			-\iot\bv\cdot\nabla\theta dxdt=\iot(\qi-\qp)\theta dxdt \ \ \mbox{for each $\theta\in L^2(0,T;\wot)$, and }\nonumber
		\end{equation}
	\begin{eqnarray}
		\lefteqn{	-\iot\phx u\pt\zeta dxdt+\iot\phx D(\bv)\nabla u\cdot\nabla\zeta dxdt+\iot\nabla u\cdot\bv\zeta dxdt}\nonumber\\
		&=&\iot\qi(\hat{u}-u)dxdt+\io u_0(x)\zeta(x,0)dx\ \ \mbox{for each $\zeta\in C^\infty(\overline{\ot})$ with $\zeta(x,T)=0$.}\nonumber
		\end{eqnarray}
	\end{enumerate}
\end{defn}
Obviously, the uniqueness in the $p$-component does not hold. If $(p, u)$ is a weak solution, so is $(p+c, u)$ for each constant $c$. Our solution to be obtained will satisfy
\begin{equation}
	\io pdx=0\ \ \mbox{for a.e. $t\in (0,T)$.}\nonumber
\end{equation}
We can easily deduce from \eqref{pro1}, \eqref{pro2}, and \eqref{pro4} that
\begin{eqnarray}
	\io\nabla u\cdot\bv\zeta dx&=&-\io u\divg\bv\zeta dx-	\io u\bv\cdot\nabla\zeta dx\nonumber\\
	&=&-\io u(\qi-\qp)\zeta dx-	\io u\bv\cdot\nabla\zeta dx \ \ \mbox{for a.e. $t\in (0,T)$}.\label{de4}
\end{eqnarray}
We will repeatedly make use of this fact.

Problems of this type have attracted a lot of attention. The first existence theorem for \eqref{pro1}-\eqref{pro6} was established in \cite{F} for the case $N=2$. This result was later generalized by Chen and Ewing \cite{CE} to the three dimensional case and with gravity effects and various boundary conditions. Amirat and Ziani \cite{AZ} studied the asymptotic behavior of the weak solution as $m$ goes to $0$ and they proved an existence assertion when $m=0$ and $N=2$, or $3$. Also see \cite{C1} for a related result. Existence with measure data was considered in \cite{DT}. We refer the reader to \cite{CD} for a numerical approach to the problem. 
In \cite{C} the author investigated regularity properties of solutions to a problem similar to ours. Unfortunately, Lemma 2 in \cite{C} does not seem to be correct.

In this paper, we will first construct a weak solution without any restrictions on the space dimensions. Our assumptions on the given data are also much weaker. To be precise, we have
\begin{thm}[Existence]\label{exi1} Let (H1)-(H6) hold. Assume:
	\begin{enumerate}
		\item[\textup{(H7)}]$\ob$ is $C^{1}$.
	\end{enumerate}  Then there is a weak solution $(p,u)$ to \eqref{pro1}-\eqref{pro6}.
	\end{thm} 
Then we turn our attention to
 the regularity properties of weak solutions. As we mentioned earlier, equation \eqref{pro3} is singular when $|\bv|$ is infinity. Moreover, the diffusive coefficient matrix $\phx D(\bv)$ there depends on $\nabla p$. Thus, cross-diffusion takes place. Another challenging feature is that the quadratic nonlinear term $\nabla p\cdot \nabla u$ is present in the equation. In spite of these difficulties, 
 we still manage to obtain a partial regularity theorem. Before we present our results, we introduce some notions and definitions. For $x_0\in\rn, t_0>0, z_0=(x_0, t_0), r>0$, we define 
$$\br=\{y\in\rn: |y-x_0|<r\},\ \ \qr=\br\times \left(t_0-\frac{1}{2}r^2, t_0+\frac{1}{2}r^2\right).$$
If $f\in L^1_{\textup{loc}}(\mathbb{R}^N)$,  we write
$$ \dashint_{\br} f(x)dx=(f)_{x_0,r}=\frac{1}{|\br|} \int_Ef(x)dx.$$
Let  $(p,u)$ be a weak solution to \eqref{pro1}-\eqref{pro6}.
A point $z_0\in\ot$ is called a regular point for the weak solution  if  for each $\ell>1$ there is a $r>0$ such that
\begin{equation}\label{in2}
	|\nabla p|\in L^\infty\left(t_0-2r^2,t_0+2r^2; L^\ell(\br)\right) .
\end{equation}
If $z_0$ is not a regular point, then we say it is a singular point.
\begin{thm}[Partial regularity] \label{par2}Let (H1)-(H7) hold and $(p,u)$ be the weak solution constructed in Theorem \ref{exi1}. Assume:
	\begin{enumerate}
		\item[\textup{(H8)}]  $\Phi(x)\in W^{1,\infty}(\Omega)$;
		\item[\textup{(H9)}] $N=2,3$, or $4$.
	\end{enumerate} 
 Then $z_0=(x_0,t_0)\in\ot$ is a regular point whenever it satisfies
	\begin{eqnarray}
		\limsup_{\rho\ra 0}\sup_{t_0-\frac{1}{2}\rho^2\leq t\leq t_0+\frac{1}{2}\rho^2}\dashint_{B_\rho(x_0)}|\nabla p|^2dx&<&\infty\ \ \mbox{and }\label{in1}\\
		\textup{osc}_{\br}K(x)&\ra& 0 \ \ \mbox{as $r\ra 0$.} \label{in3}
	\end{eqnarray}
Here and in what follows $\sup_E f$ means $\textup{ess sup}_E f$ for any measurable function $f$ defined on measurable set $E$ and $\textup{osc}_{\br}K(x)$ denotes the oscillation of $K(x)$ over $\br$, i.e.,
$$\textup{osc}_{\br}K(x)=\sup_{\br}K(x)-\inf_{\br}K(x).$$
\end{thm}

 Our result implies that low regularity of our weak solution at a single point can be improved to high regularity in a neighborhood of the point for the $p$-component. The singularity in \eqref{pro3} prevents us from obtaining similar results for the $u$ component. Since \eqref{in1} holds a.e. on $\ot$, our result implies that the singular set is discrete in the sense that it does not have any interior points in $\ot$. It would be interesting to know the parabolic Hausdorff dimension of the singular set as in \cite{LX,X2}.

There have been many papers that deal with results like \eqref{in2}. See \cite{DF} and the references therein. However, our situation here does not seem to fit into frameworks in \cite{DI,DF}. Instead, we employ an idea from \cite{X1}, which is based upon the the Fefferman-Stein inequality \cite{S}. To introduce the inequality,
let $f$ be a locally integrable function defined on $\rn$. 
 The sharp function $f^\#$, associated to  $f$, is defined by
$$ f^\#(x)=\sup_{x\in \bry, R>0}\dashint_{B_R(y)}|f(\xi)-(f)_{y,R}|d\xi.$$
If $f\in L^\ell(\rn)$, with $1<\ell<\infty$, then the following inequality holds:
\begin{equation}\label{fs}
		\|f\|_{\ell, \rn}\leq c(\ell)\|f^\#\|_{\ell,\rn}.
\end{equation}
This inequality is often called the Fefferman-Stein inequality. Its proof can be found in (\cite{S}, p.148).
We also need to consider
the (Hardy–Littlewood) maximal function of $f$, denoted by $M(f)$. It is defined to be
$$M(f)(x)=\sup_{r>0}\dashint_{B_r(x)}|f(y)|dy.$$
The well known Hardy–Littlewood maximal inequality asserts that for each $\ell\in(1, \infty]$ there is a positive number $c=c(\ell)$ such that
\begin{eqnarray}
	\|M(f)\|_{\ell, \rn}\leq c_\ell\|f\|_{\ell,\rn} \ \ \mbox{for each $f\in L^\ell(\rn)$.}\label{hl}
\end{eqnarray}
Even though the constant $c_\ell$ in the above inequality is bounded as $\ell\ra \infty$ (\cite{S1}, p.7), the constant $c(\ell)$ in \eqref{fs} must tent to infinity as $\ell\ra \infty$ because a BMO function is not necessarily bounded (\cite{S}, p.140).

This paper is organized as follows. In Section 2 we present an approximation scheme for \eqref{pro1}-\eqref{pro6}, by which an existence assertion is established. Section 3 is devoted to the proof of partial regularity.

Throughout the paper the letter $c$ will be used to denote a generic positive number whose value can be computed from the given data at least in theory. H\"{o}lder's inequality is often used without acknowledgment.

\section{Existence}
In this section, we first design an approximation scheme. Then an existence assertion is established for the approximate problems. At the end of the section, we justify passing to the limit in the approximate problems to obtain a weak solution to \eqref{pro1}-\eqref{pro6}.

To remove the possible singularity in \eqref{pro3}, we introduce the function
$$\ek(s)=\left\{\begin{array}{ll}
	k&\mbox{if $s>k$},\\
	s&\mbox{if $s\leq k$},\ \ k=1,2,\cdots.
\end{array}\right.$$
Then we approximate $D(x)$ by
$$D_k(x)=(a\ek(|x|)+m)I+(b-a)\ek(|x|)\frac{x\otimes x}{|x|^2}, \ \ x\in\rn.$$
Obviously, $D_k(x)$ is a Lipschitz function of $x$ (even uniformly in $k$). Moreover,
\begin{equation}\label{add1}
	m|\xi|^2\leq D_k(x)\xi\cdot\xi=(a\ek(|x|)+m)|\xi|^2+(b-a)\ek(|x|)\frac{(x\cdot\xi)^2}{|x|^2}\leq (b\ek(|x|)+m)|\xi|^2
\end{equation}
for each $\xi\in\rn$. 
By (H1), we can pick a sequence
\begin{equation}\label{add6}
	\{\phkx\}\subset W^{1,\infty}(\Omega)
\end{equation}  with the property
\begin{equation}\label{add2}
	\phkx\ra \phx\ \ \mbox{strongly in $W^{1,2}(\Omega)$ and $\inf_{\overline{\Omega}}\phkx\geq\lambda_0$ for each $k$}
\end{equation}
Similarly,  (H2) asserts that we can choose 
\begin{equation}\label{add7}
	\{K_k(x)\}\subset\left(C(\overline{\Omega})\right)^{N\times N}
\end{equation} so that
\begin{eqnarray}
	K_k(x)&\ra& K(x)\ \ \mbox{strongly in $\left(L^\ell(\Omega)\right)^{N\times N}$ for each $\ell\geq 1$, }\label{add3}\\
		K_k(x)&\ra& K(x)\ \ \mbox{ weak$^*$ in $\left(L^\infty(\Omega)\right)^{N\times N}$, and}\label{add4}\\
	\xi\cdot K_k(x)\xi&\geq& c_0|\xi|^2\ \ \mbox{for each $\xi\in\rn, x\in\Omega$, and $k$.}\label{add5}
\end{eqnarray}
We form our approximate problems as follows:
\begin{eqnarray}
	\divg\bv&=&\qi-\qp\ \ \mbox{in $\ot$,}\label{apro1}\\
	\bv&=&-\epck\nabla p\ \ \mbox{in $\ot$,}\label{apro2}\\
	\phkx\pt u-\divg\left(\phkx D_k(\bv)\nabla u\right)+\nabla u\cdot\bv 
	&=&\qi\left(\hat{u}-u\right) \ \mbox{in $\ot$},\label{apro3}\\
	K_k(x)\nabla p\cdot\mathbf{n}&=&0\ \ \mbox{on $\Sigma_T$,}\label{apro4}\\
	D_k(\bv)\nabla u\cdot\mathbf{n}&=&0\ \ \mbox{on $\Sigma_T$,}\label{apro5}\\
	u(x,0)&=&u_0(x)\ \ \mbox{on $\Omega$},\label{apro6}
\end{eqnarray}
 Observe that equation \eqref{apro3} is now uniformly parabolic for each fixed $k$. 
\begin{thm}\label{existence} 	Assume (H1)-(H7) hold. Let $D_k(x), \phkx, K_k(x)$ be given as before.
 Then for each fixed $k$ there exists a weak solution $(p, u)$ to \eqref{apro1}-\eqref{apro6} with
	\begin{equation}
		p\in L^\infty(0,T; W^{1,\ell}(\Omega))\ \ \mbox{for each $\ell>1$ and $u\in C^{\beta,\frac{\beta}{2}}(\overline{\ot})$ for some $\beta\in(0,1).$ }
	\end{equation}
\end{thm}
\begin{proof}
	A weak solution will be obtained via the Leray-Schauder fixed point theorem (\cite{GT}, p.280). To this end, we define an operator $\mathbb{T}$ from $C(\overline{\ot})$ into itself as follows: Let $u\in C(\overline{\ot})$ be given. For a.e. $t\in(0,T)$ we solve the elliptic boundary value problem
\begin{eqnarray}
	-\divg\left(\epck\nabla p\right)&=&\qi-\qp\ \ \mbox{in $\Omega$,}\label{pe1}\\
	K_k(x)\nabla p\cdot\mathbf{n}&=&0\ \ \mbox{on $\ob$,}\label{pe2}\\
	\io p dx&=&0.\label{pe3}
\end{eqnarray}
The existence of a unique weak solution $p$ in $\wot$ to the above problem is standard due to \eqref{add5}, (H5), and (H3). Use (H5) again to obtain that $\frac{1}{\mu(u)}$ is a continuous function on $\overline{\Omega}$ for each $t\in[0,T]$.
This together with \eqref{add7} asserts
that the entries of $\epck$ still lie in $C(\overline{\Omega})$.  In view of (H7), we are in a position to invoke a  well known result in \cite{ADN1,ADN2}. 
Upon doing so, we conclude
  that for each $\ell>1$ there is a positive number $c$ such that
\begin{eqnarray}
	\|\nabla p\|_{\ell, \Omega}\leq c\|\qi-\qp\|_{\ell,\Omega}\ \ \mbox{for a.e $t\in(0,T)$}.\label{wl}
\end{eqnarray}
(In fact, our assumptions here are much stronger than what is needed for \eqref{wl} to hold. See the most recent result in \cite{DK} in this direction.)
Define $\bv$ as in \eqref{apro2}. Consider the problem
\begin{eqnarray}
	\phkx\pt \varphi-\divg\left(\phkx D_k(\bv)\nabla \varphi-\varphi\bv\right)+\qp \varphi&=&\qi\hat{u}\ \mbox{in $\ot$},\label{vp1}\\
		D_k(\bv)\nabla \varphi\cdot\mathbf{n}&=&0\ \ \mbox{on $\Sigma_T$,}\label{vp2}\\
	\varphi(x,0)&=&u_0(x)\ \ \mbox{on $\Omega$},\label{vp3}
\end{eqnarray}
Note that by virtue of \eqref{de4} we see that \eqref{vp1} is equivalent to \eqref{apro3}.
\begin{clm}\label{clm1}
	There exists a unique weak solution $\varphi$ to the preceding problem in the space $C^{\beta,\frac{\beta}{2}}(\overline{\ot})\cap L^2(0,T;\wot)$ for some $\beta\in (0,1)$. Moreover,
	\begin{equation}\label{vp8}
		0\leq\varphi\leq 1.
	\end{equation}
\end{clm}
\begin{proof}[Proof of the claim]
	Since we have \eqref{wl}, \eqref{add1}, \eqref{add2}, \eqref{add6}, and (H3), the classical existence and regularity theory for linear parabolic equations are applicable \cite{LSU}, from which follows the existence part. In particular, H\"{o}lder continuity of $\vp$ can be inferred from Theorem 10.1 in \cite{LSU}. (Our Neumann boundary condition here is not a problem due to (H7).)
	As for the uniqueness, it is enough for us to show that the only solution to the problem 
\begin{eqnarray}
	\phkx\pt \varphi-\divg\left(\phkx D_k(\bv)\nabla \varphi-\varphi\bv\right)+\qp \varphi&=&0\ \mbox{in $\ot$},\label{vp4}\\
	D_k(\bv)\nabla \varphi\cdot\mathbf{n}&=&0\ \ \mbox{on $\Sigma_T$,}\label{vp5}\\
	\varphi(x,0)&=&0\ \ \mbox{on $\Omega$}\label{vp6}
\end{eqnarray}
is the trivial solution $0$. To this end, we use $\varphi$ as a test function in \eqref{vp4} to derive
\begin{equation}\label{vp7}
	\frac{1}{2}\frac{d}{dt}\io\phkx\vp^2dx+\io\phkx\dk\nabla\vp\cdot\nabla\vp dx-\io\vp\nabla\vp\cdot\bv dx+\io\qp\vp^2dx= 0.
\end{equation}
In view of \eqref{pe1} and \eqref{apro2}, we have
\begin{equation}\label{vp9}
-\io\vp\nabla\vp\cdot\bv dx=-\frac{1}{2}\io\nabla \vp^2\cdot\bv dx=\frac{1}{2}\io(\qi-\qp)\vp^2dx.	
\end{equation}
Substitute this into \eqref{vp7} and integrate to derive
$$\io\phkx\vp^2(x,t)dx\leq 0.$$
Thus, $\vp\equiv 0$.

To establish \eqref{vp8}, we use $-\vp^-$ as a test function in \eqref{vp1} to get
\begin{eqnarray}
	\lefteqn{	\frac{1}{2}\frac{d}{dt}\io\phkx(\vp^-)^2dx+\io\phkx\dk\nabla\vp^-\cdot\nabla\vp^- dx}\nonumber\\
	&&-\io\vp^-\nabla\vp^-\cdot\bv dx+\io\qp(\vp^-)^2dx= -\io\qi\hat{u}\vp^-dx\leq 0.\label{vp10}
\end{eqnarray}
By a calculation similar to \eqref{vp9}, we have
$$-\io\vp^-\nabla\vp^-\cdot\bv dx=\frac{1}{2}\io(\qi-\qp)(\vp^-)^2dx.$$
This together with \eqref{vp10} and \eqref{add2} implies
$$\lambda_0\io(\vp^-)^2(x,t)dx\leq \io\phkx(\vp^-)^2(x,t)dx\leq \io\phkx(u_0^-)^2(x)dx=0,$$
from whence follows
$$\vp\geq 0\ \ \mbox{in $\ot$}.$$
Set
$$w=1-\vp.$$
We can easily verify that
$$	\phkx\pt w-\divg\left(\phkx D_k(\bv)\nabla w-w\bv\right)+\qp w=\qi(1-\hat{u})\geq 0\ \mbox{in $\ot$}.$$
Repeat our earlier argument to get $w\geq 0$. This completes the proof of the claim.
\end{proof}
Equipped with this claim,  we can define
$$\mathbb{T}(u)=\vp.$$
Observe from the constructions of $D_k(x), \phkx$, and $K_k(x)$ and \eqref{wl} that the H\"{o}lder norm of $\vp$ depends only on $k$ and other given data in \eqref{vp1}-\eqref{vp3}. 
Thus, the range of $\mathbb{T}$ is bounded in $ C^{\beta,\frac{\beta}{2}}(\overline{\ot})$, where $\beta$ is given as in  Claim \ref{clm1}.  That is, 
 $\mathbb{T}$ maps bounded sets in $C(\overline{\ot})$ into precompact ones. To see that $\mathbb{T}$ is also continuous, we assume
$$u_m\ra u \ \ \mbox{strongly in $C(\overline{\ot})$.}$$
Replace $u$ by $u_m$ in \eqref{pe1}-\eqref{pe3} and denote the resulting solution by $p_m $.
Note that
$$-\divg\left(\frac{1}{\mu(u_n)}K_k(x)\nabla p_n-\frac{1}{\mu(u_m)}K_k(x)\nabla p_m\right)=0\ \ \mbox{in $\Omega$.}$$
Use $ p_n-p_m$ as a test function in the above equation to derive
\begin{eqnarray}
\lefteqn{	\io \frac{1}{\mu(u_n)}K_k(x)\nabla( p_n-p_m)\cdot\nabla( p_n-p_m)dx}\nonumber\\
&=&\io\left(\frac{1}{\mu(u_m)}-\frac{1}{\mu(u_n)}\right)K_k(x)\nabla p_m\cdot\nabla( p_n-p_m)dx.\label{rvm18}
\end{eqnarray}
This combined with \eqref{add5}, \eqref{add4}, and (H5) implies
\begin{equation}
	\sup_{0\leq t\leq T}\io|\nabla(p_n-p_m)|^2dx\leq c\|\mu(u_n)-\mu(u_m)\|_{\infty,\ot}\ra 0 \ \ \mbox{as $n,m\ra\infty$}.\nonumber
\end{equation}
Let $\vp_m=\mathbb{T}(u_m)$. That is, 
\begin{eqnarray}
	\bv_m&=&-\frac{1}{\mu(u_m)}K_k(x)\nabla p_m,\label{vp18}\\
	\phkx\pt \vp_m-\divg\left(\phkx D_k(\bv_m)\nabla \vp_m-\vp_m\bv_m\right)+\qp \vp_m&=&\qi\hat{u}\ \mbox{in $\ot$},\label{vp11}\\
	D_k(\bv_m)\nabla \vp_m\cdot\mathbf{n}&=&0\ \ \mbox{on $\Sigma_T$,}\label{vp12}\\
	\vp_m(x,0)&=&u_0(x)\ \ \mbox{on $\Omega$}.\label{vp13}
\end{eqnarray}
We infer from \eqref{wl} and Claim \ref{clm1} that
\begin{eqnarray}
&&\mbox{$	\{v_m\} $ is bounded in $L^\infty(0,T; W^{1,\ell}(\Omega))$ for each $\ell\geq 1$},\nonumber\\
&&\mbox{$\{\vp_m\}$ is bounded in $C^{\beta,\frac{\beta}{2}}(\overline{\ot})\cap L^2(0,T;\wot)$ for some $\beta\in (0,1)$.}\nonumber
\end{eqnarray}
We can pass to the limit in \eqref{vp11}-\eqref{vp13} at least along a subsequence. The whole sequence also converges due to the uniqueness assertion in Claim \ref{clm1}.

We still need to show that there is a positive number $c$ such that
$$\max_{\ot}|u|\leq c$$
for all $u\in C(\overline{\ot})$ and all $\sigma\in (0,1)$ satisfying $u=\sigma\mathbb{T}(u)$. This is an easy consequence of \eqref{vp8}. The proof of Theorem \ref{existence} is complete.
\end{proof}
\begin{proof}[Proof of Theorem \ref{exi1}] Denote by $(p_k,u_k)$ the weak solution to \eqref{apro1}-\eqref{apro6}. That is, we replace $ \bv, p, u$ in \eqref{apro1}-\eqref{apro6} by $\bv_k, p_k, u_k$, respectively. The challenge here is that many of the previous estimates do not hold uniformly in $k$. In particular,  \eqref{wl} and the boundedness of $D_k(\bv_k)$  are no longer available to us.
	 Fortunately, we can still collect enough estimates to justify passing to the limit in the weak formulation of problem \eqref{apro1}-\eqref{apro6}.
	  First, 
	  Poincar\'{e}'s inequality together with \eqref{pe3} implies
	  \begin{eqnarray}
	  	\|p_k\|_{2,\Omega}\leq c\|\nabla p_k\|_{2,\Omega}.\nonumber
	  \end{eqnarray}
  Use $p_k$ as a test function in \eqref{apro1} and apply the above inequality to deduce
	\begin{equation}
		\|\nabla p_k\|_{2,\Omega}\leq c	\|\qi-\qp\|_{2,\Omega}.\nonumber
	\end{equation}
By employing an argument similar to the proof of Theorem 8.15 in (\cite{GT}, p.189), we can conclude that for each $\ell>\frac{N}{2}$ there is a positive number $c$ such that
\begin{equation}\label{AA}
	\|p_k\|_{\infty,\Omega}\leq c\left(\|p_k\|_{2,\Omega}+\|\qi-\qp\|_{\ell,\Omega}\right)\leq c\|\qi-\qp\|_{\ell,\Omega}.
\end{equation}
In fact, since we have \eqref{pe3} the proof there can be carried over here in a straightforward manner.
Moreover, a result of \cite{GR} asserts that there is a $s>2$ such that
\begin{equation}\label{vp15}
		\|\nabla p_k\|_{s,\Omega}\leq c	\|\qi-\qp\|_{s,\Omega}\ \ \mbox{for a.e. $t\in(0,T)$}.
\end{equation}
This inequality is often called the  Meyers type estimate in the literature due to the first result in this direction by Meyers in \cite{M}.
As for $\{u_k\}	$, we have from Claim \ref{clm1} that
\begin{equation}\label{vp16}
	0\leq u_k\leq 1 \ \ \mbox{in $\ot$.}
\end{equation}
Use $u_k$ as a test function in \eqref{apro3} to derive
\begin{eqnarray}
	\lefteqn{\frac{1}{2}\frac{d}{dt}\io \phkx u_k^2dx+\io\phkx D_k(\bv_k)\nabla u_k\cdot\nabla u_kdx}\nonumber\\
	&&+\io u_k\nabla u_k\cdot\bv_kdx+\io\qi u_k^2dx=\io\qi\hat{u} dx.\label{vp14}
\end{eqnarray}
As before, we have
$$	\io u_k\nabla u_k\cdot\bv_kdx=-\frac{1}{2}\io u_k^2(\qi-\qp)dx.$$
This together with \eqref{vp14} implies
\begin{equation}\label{vp20}
	\sup_{0\leq t\leq T}\io \phkx u_k^2dx+\iot(1+|\bv_k|)|\nabla u_k|^2dxdt\leq c.
\end{equation}
It immediately follows that
\begin{equation}\label{rvp20}
	u_k\ra u\ \ \mbox{weakly in $L^2(0,T; W^{1, 2}(\Omega))$ at least along a subsequence.}
\end{equation}
Let $s$ be given as in \eqref{vp15}. Then
$$\ell_0\equiv\frac{2s}{s+2}\in (1,2).$$
We derive from \eqref{vp15} and \eqref{vp20} that
\begin{eqnarray}
\lefteqn{	\iot(1+|\bv_k|)^{\ell_0}|\nabla u_k|^{\ell_0}dxdt}\nonumber\\
&\leq&\sup_{0\leq t\leq T}\left(\io(1+|\bv_k|)^{\frac{\ell_0}{2-\ell_0}}dx\right)^{\frac{2-\ell_0}{2}}\left(\iot(1+|\bv_k|)|\nabla u_k|^2dxdt\right)^{\frac{\ell_0}{2}}\leq c.\label{rvp17}
\end{eqnarray}	
Combining this with \eqref{apro3} yields	
\begin{equation}
	\mbox{ $\{\pt\left(\phkx u_k\right)\}$ is bounded in $L^{\frac{\ell_0}{\ell_0-1}}\left(0,T;\left(W^{1, \frac{\ell_0}{\ell_0-1}}(\Omega)\right)^*\right)$.}\nonumber
\end{equation}
By a calculation similar to \eqref{rvp17}, we 
can derive from  \eqref{add2} and \eqref{vp20} that 
\begin{equation}\label{vp19}
	\mbox{$\{\phkx u_k\}$ is bounded in $L^2(0,T; W^{1, \frac{N}{N-1}}(\Omega))$.}\nonumber
\end{equation}
Here we have assumed $N>2$. If $N=2$, we can replace $\frac{N}{N-1}$ by any number in the interval $(1,2)$.
We can choose $s$ so close to $2$ that
$$\frac{\ell_0}{\ell_0-1}=\frac{2s}{s-2}>N.$$
Subsequently,
$$L^{\frac{N}{N-1}}(\Omega)\subset W^{1,\frac{N}{N-1}}(\Omega)\subset\left(W^{1, \frac{\ell_0}{\ell_0-1}}(\Omega)\right)^*.$$
Clearly, the first inclusion is compact, while the second inclusion is continuous.
We are in a position to invoke a version of Lions-Aubin \cite{SI} to obtain
	\begin{eqnarray}\label{vp17}
		u_k\ra u\ \ \mbox{a.e. on $\ot$}
	\end{eqnarray}
(pass to a subsequence if necessary).
This together with \eqref{vp16} implies that
\begin{eqnarray}
	u_k\ra u\ \ \mbox{strongly in $L^q(\ot)$ for each $q>1$.}\nonumber
\end{eqnarray}
With the aid of \eqref{rvm18} and \eqref{vp15}, we calculate
\begin{eqnarray}
\lefteqn{\io|\nabla(p_n-p_m)|^2dx}\nonumber\\
&\leq& c\left\|\frac{1}{\mu(u_n)}K_n(x)-\frac{1}{\mu(u_m)}K_m(x)\right\|_{\frac{s}{s-2},\Omega}\|\nabla p_m\|_{s,\Omega}\|\nabla (p_n-p_m)\|_{s,\Omega}\nonumber\\
&\leq& c\|\mu(u_n)-\mu(u_m)\|_{\frac{s}{s-2},\Omega}+c\|K_n(x)-K_m(x)\|_{\frac{s}{s-2},\Omega}.\nonumber
\end{eqnarray}
By the continuity of $\mu$, \eqref{vp17}, and \eqref{add3}, we have
$$\iot|\nabla(p_n-p_m)|^2dxdt\leq c\|\mu(u_n)-\mu(u_m)\|_{\frac{s}{s-2},\ot}+c\|K_n(x)-K_m(x)\|_{\frac{s}{s-2},\Omega}\ra 0\ \ \mbox{as $m,n\ra\infty$.}$$
Recall that $\io p_k dx=0$ for all $k$. Consequently,
$$\io|p_n-p_m|^2dx\leq c\io|\nabla(p_n-p_m)|^2dx.$$
We may assume
$$p_k\ra p\ \ \mbox{strongly in $L^2(0,T;\wot)$ and a.e on $\ot$}.$$
It immediately follows from \eqref{apro2} that
\begin{equation}
	\bv_k\ra-\frac{1}{\mu(u)}K(x)\nabla p\equiv \bv\ \ \mbox{strongly in $\left(L^2(\ot)\right)^N$.}\nonumber
\end{equation}
Recall from (H4) and the definition of $D_k$ that
$$|D_k(x)-D_k(y)|\leq c|x-y|,\ \  |D_k(\bv_k)|\leq b|\bv_k|+m.$$
Thus, $D_k(\bv_k)\ra D(\bv)$ strongly in $\left(L^2(\ot)\right)^{N\times N}$. This, combined with \eqref{rvp20},  yields
$$D_k(\bv_k)\nabla u_k\ra D(\bv)\nabla u\ \ \mbox{weakly in $L^{\ell_0}(\ot)$}.$$
We are ready to pass to the limit in \eqref{apro1}-\eqref{apro6}. The proof is complete.
\end{proof}	

\section{Partial regularity}
In this section we offer the proof of Theorem \ref{par2}. The idea is to show that if $z_0\in\ot$ is such that \eqref{in1}  holds then $u$ is continuous at $z_0$. This combined with \eqref{in3} is enough to raise the integrability of $|\nabla p|$ in a small neighborhood of $z_0$. We divide the proof into several lemmas.
\begin{lem}\label{reh}
	There exist $s>2$ and $c>0$ such that
	\begin{equation}
		\left(\aibrh|\nabla p|^sdx\right)^{\frac{1}{s}}\leq c\left[\left(\aibr|\nabla p|^2dx\right)^{\frac{1}{2}}+r\left(\aibr|\qi-\qp|^sdx\right)^{\frac{1}{s}}\right]\nonumber
	\end{equation}
for all balls $\br$ with $B_{2r}(x_0)\subset\Omega$ and a.e. $t\in (0,T)$.
\end{lem}
Note that our notion of a weak solution does not imply that $u$ is continuous. 
Thus, the elliptic coefficients in \eqref{pe1} are merely bounded besides the uniform ellipticity condition.
Lemma \ref{reh} is just a local version of the Meyers type estimate we mentioned earlier. For completeness, we offer a proof here.
\begin{proof}
	Let $x_0\in\Omega$ and $r>0$ be given as in the lemma. We introduce the change of the space variables
	$$\xi=\frac{x-x_0}{r}$$
	and denote by $\ptl(\xi, t)$ the composition $p(x_0+r\xi, t)$. We can easily verify that $\ptl$ satisfies the equation
	\begin{equation}\label{ppe1}
		-\divg\left(J(\xi, t)\nabla \ptl\right)=r^2F(\xi, t)\ \ \mbox{in $B_2(0)$},
	\end{equation}
	where
	$$J(\xi, t)=\left.\epc\right|_{x=x_0+r\xi},\ \ F(\xi, t)=\left.(\qi-\qp)\right|_{x=x_0+r\xi}.$$
	For each $y\in B_2(0)$ and $R>0$ such that $\bry\subset B_2(0)$ we pick a function $\zeta\in C^\infty_0(\bry)$ with the properties 
	$$\zeta=1\ \ \mbox{on $\bryh$, $0\leq\zeta\leq 1$ on $\bry$, and $|\nabla\zeta|\leq \frac{c}{R}$ on $\bry$.}$$
	Use $\zeta^2(\ptl-\ptl_{y,R})$ as a test function in \eqref{ppe1} and then apply (H2), (H5), and the Sobolev embedding theorem  to get
	\begin{eqnarray}
		\lefteqn{\ibry J \nabla \ptl\cdot\nabla \ptl\zeta^2d\xi }\nonumber\\
		&=&-2\ibry J \nabla \ptl\cdot\nabla \zeta (\ptl-\ptl_{y,R})\zeta d\xi +r^2\ibry F\zeta^2(\ptl-\ptl_{y,R})d\xi \nonumber\\
		&\leq& \ve\ibry\zeta^2|\nabla \ptl|^2d\xi +\frac{c}{\ve R^2}\ibry(\ptl-\ptl_{y,R})^2d\xi \nonumber\\
		&&+r^2\|F\|_{\frac{2N}{N+2},\bry}\|\zeta(\ptl-\ptl_{y,R})\|_{\frac{2N}{N-2},\bry}\nonumber\\
		&\leq& 2\ve\ibry\zeta^2|\nabla \ptl|^2d\xi +\left(\frac{c}{\ve R^2}+\frac{c\ve}{R^2}\right)\ibry(\ptl-\ptl_{y,R})^2d\xi +\frac{cr^4}{\ve}\|F\|_{\frac{2N}{N+2},\bry}^2,\ \ \ve>0.\label{haha1}
	\end{eqnarray}
Here we have assumed $N>2$. We will do so whenever we must use the Sobolev embedding theorem. If $N=2$, we just need to do the following modification
\begin{eqnarray}
	\left(\io|f|^\ell dx\right)^{\frac{1}{\ell}}&\leq &c \left(\left(\io|\nabla f|^{\frac{2\ell}{\ell+2}}dx\right)^{\frac{\ell+2}{2\ell}}+\left(\io| f|^{\frac{2\ell}{\ell+2}}dx\right)^{\frac{\ell+2}{2\ell}}\right)\nonumber\\
	&\leq& c  \left(\left(\io|\nabla f|^{2}dx\right)^{\frac{1}{2}}+\left(\io| f|^{2}dx\right)^{\frac{1}{2}}\right)\ \ \mbox{for each $\ell\geq 2$}.\nonumber
\end{eqnarray}
Return to our proof. Use (H2) and (H5) again in \eqref{haha1}, choose $\ve$ suitably small, and take into account Poincar\'{e}'s inequality on balls (\cite{EG}, p.141) to derive
\begin{eqnarray}
	\ibryh|\nabla \ptl|^2d\xi &\leq&\frac{c}{R^2}\ibry(\ptl-\ptl_{y,R})^2d\xi +cr^4\|F\|_{\frac{2N}{N+2},\bry}^2\nonumber\\
	&\leq& cR^N\left(\dashint_{B_R(y)}|\nabla \ptl|^{\frac{2N}{N+2}}d\xi \right)^{\frac{N+2}{N}}+ cR^{N+2}\left(\dashint_{B_R(y)}|r^2F|^{\frac{2N}{N+2}}d\xi \right)^{\frac{N+2}{N}}.\label{rhi1}
\end{eqnarray}
Set
$$g=|\nabla \ptl|^{\frac{2N}{N+2}},\ \ f=|r^2F|^{\frac{2N}{N+2}}.$$
Then we can write \eqref{rhi1} as
$$ \dashint_{B_{\frac{R}{2}}(y)}g^{\frac{N+2}{N}}d\xi \leq c\left(\dashint_{B_R(y)} gd\xi \right)^{\frac{N+2}{N}}+cR^2 \dashint_{B_R(y)}f^{\frac{N+2}{N}}d\xi .$$
Now we are in a position to apply Gehring’s lemma (Proposition 1.1 in \cite{G}). Upon doing so, we obtain that there is a $\ell>\frac{N+2}{N}$ such that
$$\left(\dashint_{B_{\frac{1}{2}}(0)} g^\ell d\xi \right)^{\frac{1}{\ell}}\leq c\left[\left(\dashint_{B_1(0)} g^{\frac{N+2}{N}}d\xi \right)^{\frac{N}{N+2}}+\left(\dashint_{B_1(0)} f^\ell d\xi \right)^{\frac{1}{\ell}}\right].$$
 Take
 $$s=\frac{2N\ell}{N+2}.$$
 Return to the $x$-variable and keep in mind the fact that $\nabla\ptl=r\nabla p(x_0+r\xi)$ to get the lemma.
\end{proof}
This lemma combined with \eqref{qcon} implies that for each $\ell_1\geq s$ we have
	\begin{eqnarray}
		\left(\aibrh|\nabla p|^sdx\right)^{\frac{1}{s}}&\leq& c\left[\left(\aibr|\nabla p|^2dx\right)^{\frac{1}{2}}+r\left(\aibr|\qi-\qp|^{\ell_1}dx\right)^{\frac{1}{\ell_1}}\right]\nonumber\\
		 &\leq&c\left(\aibr|\nabla p|^2dx\right)^{\frac{1}{2}}+cr^{1-\frac{N}{\ell_1}}.\label{haha2}
	\end{eqnarray}
	for all balls $\br$ with $B_{2r}(x_0)\subset\Omega$.

	Let $z_0=(x_0,t_0)\in\ot$ be given. For each $r>0$ with
	$\qr\subset\ot$
	we define
	$$M_r=\sup_{\qr}u, \ \ m_r=\inf_{\qr}u, \ \ \omega_r=M_r-m_r.$$

\begin{lem}\label{sob}Let $z_0, r$ be given as before. Assume
	\begin{equation}\label{vpe3}
	\limsup_{\rho\ra 0}	\sup_{t_0-\frac{1}{2}\rho^2\leq t\leq t_0+\frac{1}{2}\rho^2}\dashint_{B_\rho(x_0)}|\nabla p|^2dx<\infty.
	\end{equation}Then for each $\ell>\frac{N}{2}$ there exist two numbers $s_1>2$ and $\lambda_1>0$ such that either
	\begin{equation}\label{po3}
		\left|\br\cap\left\{u(x,t)>M_r-\frac{\omega_r}{2^{s_1}}\right\}\right|\leq\frac{7}{8}\left|\br\right|\ \ \textup{ for each $t\in\itz$,}
	\end{equation}
or
\begin{equation}\label{po4}
	\omega_r\leq \lambda_1r^{2-\frac{N}{\ell}}.
\end{equation}	
\end{lem}
\begin{proof}Our proof here is largely inspired by Chapter III in \cite{D}. Obviously, either
	\begin{equation}\label{po1}
			\left|\br\cap\left\{u\left(x,t_0-\frac{1}{2}r^2\right)>M_r-\frac{\omega_r}{2}\right\}\right|\leq\frac{1}{2}\left|\br\right|, \ \ \mbox{or}
	\end{equation}
\begin{equation}\label{po2}
		\left|\br\cap\left\{M_r+m_r-u\left(x,t_0-\frac{1}{2}r^2\right)\geq M_r-\frac{\omega_r}{2}\right\}\right|\leq\frac{1}{2}\left|\br\right|
\end{equation}
Assume that \eqref{po1} holds. Otherwise, use the function $M_r+m_r-u$ in the subsequent proof instead of $u$.
 
Let $L>1$ be selected as below. Introduce the function
 $$\varphi(u)=\ln^+\left(\frac{H}{H-\left(u-M_r+\frac{\omega_r}{2}\right)^++\frac{\omega_r}{2^{1+L}}}\right),$$
 where
 $$H=\sup_{\qr}\left(u-M_r+\frac{\omega_r}{2}\right)^+.$$
 For each $\gamma\in(0,1)$ we choose a cutoff function $\zeta(x)\in C^\infty_0(\br)$ with the property
 $$0\leq \zeta(x)\leq 1,\ \ \zeta(x)=1\ \ \mbox{on $B_{(1-\gamma)r}(x_0)$, and $|\nabla\zeta(x)|\leq\frac{c}{r\gamma }$.}$$
 Note that 
 \begin{equation}\label{vpe22}
 	\left[\vp^2(u)\right]^{\prime\prime}=2(1+\vp(u))\left[\vp^\prime(u)\right]^2.\nonumber
 \end{equation}
 We may use $2\Phi^{-1}(x)\varphi^\prime(u)\varphi(u)\zeta^2$ as a test function in \eqref{pro3} to derive
 \begin{eqnarray}
 	\lefteqn{\frac{d}{dt}\ibr\varphi^2(u)\zeta^2dx+2\ibr D(\bv)\nabla\varphi(u)\cdot \nabla\varphi(u)\zeta^2dx}\nonumber\\
 	&&+2\ibr D(\bv)\nabla u\cdot \nabla u\vp(u)\left[\vp^\prime(u)\right]^2\zeta^2dx\nonumber\\
 	&=&-4\ibr D(\bv) \nabla\varphi(u)\cdot \nabla\zeta \varphi(u)\zeta dx+2\ibr D(\bv) \nabla\varphi(u)\cdot \nabla\ln\phx \varphi(u)\zeta^2 dx\nonumber\\
 	&&-2\ibr\Phi^{-1}(x)\nabla\varphi(u)\cdot\bv\varphi(u)\zeta^2dx+2\ibr\qi(\hat{u}-u)\Phi^{-1}(x)\varphi^\prime(u)\varphi(u)\zeta^2dx\nonumber\\
 	&\equiv& I_1+I_2+I_3+I_4.\label{vpe1}
 \end{eqnarray}
We can easily check
\begin{equation}\label{vpe2}
0\leq \varphi(u)\leq L\ln2,\ \ 0\leq \varphi^\prime(u)\leq\frac{2^{1+L}}{\omega_r}.
\end{equation}
Equipped with these and \eqref{ellip}, we can estimate the third term from the left in \eqref{vpe1} as follows:
\begin{eqnarray}
	2\ibr D(\bv)\nabla u\cdot\nabla u\left(\varphi^{\prime}(u)\right)^2\varphi(u)\zeta^2dx
	&=&	2\ibr D(\bv)\nabla\varphi(u)\cdot\nabla\varphi(u)\varphi(u)\zeta^2dx\nonumber\\
	&\geq&2\min\{a,m\}\ibr(1+|\bv|)|\nabla\varphi(u)|^2\varphi(u)\zeta^2dx.\nonumber
\end{eqnarray}
On the other hand, by \eqref{ddef}, we have
\begin{eqnarray}
	I_1&\leq& 4\max\{m,b\}\ibr(1+|\bv|)|\nabla\varphi(u)|\varphi(u)|\nabla\zeta|\zeta dx\nonumber\\
	&\leq & \ve\ibr(1+|\bv|)|\nabla\varphi(u)|^2\varphi(u)\zeta^2dx+\frac{c}{\ve(\gamma r)^2}\ibr(1+|\bv|)\varphi(u)dx,
	 \ \ \ve>0.\nonumber
\end{eqnarray}
Similarly,
\begin{eqnarray}
	I_2&\leq&\ve\ibr(1+|\bv|)|\nabla\varphi(u)|^2\varphi(u)\zeta^2dx+\frac{c}{\ve}\ibr(1+|\bv|)|\nabla\ln\phx|^2\varphi(u)dx,\nonumber\\
	I_3&\leq&\ve\ibr|\bv||\nabla\varphi(u)|^2\varphi(u)\zeta^2dx+\frac{c}{\ve}\ibr|\bv|\left|\Phi^{-1}(x)\right|^2\varphi(u)dx.\nonumber
\end{eqnarray}
Use the preceding estimates in \eqref{vpe1}, choose $\ve$ suitably small in the resulting inequality, and thereby obtain
\begin{eqnarray}
	\frac{d}{dt}\ibr\varphi^2(u)\zeta^2dx&\leq&\frac{c}{(\gamma r)^2}\ibr(1+|\bv|)\varphi(u)dx+c\ibr(1+|\bv|)|\nabla\ln\phx|^2\varphi(u)dx\nonumber\\
	&&+c\ibr|\bv|\left|\Phi^{-1}(x)\right|^2\varphi(u)dx+2\ibr\qi(\hat{u}-u)\Phi^{-1}(x)\varphi^\prime(u)\varphi(u)\zeta^2dx.\nonumber
\end{eqnarray}
For $t\in \itz$ we integrate the above inequality over $(t_0-\frac{1}{2}r^2,t)$ and then apply (H8), (H3), and \eqref{vpe2} to derive
\begin{eqnarray}
\lefteqn{	\int_{B_{(1-\gamma)r}(x_0)}\varphi^2(u(x,t))dx}\nonumber\\
&\leq&	\int_{B_{r}(x_0)}\varphi^2\left(u\left(x,t_0-\frac{1}{2}r^2\right)\right)dx+\frac{cL}{(\gamma r)^2}\iqr(1+|\bv|)dxd\tau\nonumber\\
	&&+cL\iqr(1+|\bv|)dxd\tau+\frac{cL2^{1+L}}{\omega_r}\iqr \qi dxd\tau\nonumber\\
	&\leq&	\int_{B_{r}(x_0)}\varphi^2\left(u\left(x,t_0-\frac{1}{2}r^2\right)\right)dx+\frac{cL}{\gamma ^2}\sup_{ t_0-\frac{1}{2}r^2\leq t\leq t_0+\frac{1}{2}r^2}\ibr(1+|\bv|)dx\nonumber\\
	&&+cLr^2\sup_{ t_0-\frac{1}{2}r^2\leq t\leq t_0+\frac{1}{2}r^2}\iqr(1+|\bv|)dx+\frac{cr^2L2^{1+L}}{\omega_r}\sup_{ t_0-\frac{1}{2}r^2\leq t\leq t_0+\frac{1}{2}r^2}\ibr \qi dx.\label{vpe4}
\end{eqnarray}
By \eqref{po1} and the definition of $\varphi(u)$, we have
\begin{eqnarray}
	\int_{B_{r}(x_0)}\varphi^2\left(u\left(x,t_0-\frac{1}{2}r^2\right)\right)dx&=&\int_{\br\cap\left\{u(x,t_0-\frac{1}{2}r^2)>M_r-\frac{\omega_r}{2}\right\}}\varphi^2\left(u\left(x,t_0-\frac{1}{2}r^2\right)\right)dx\nonumber\\
	&\leq&\frac{L^2\ln^22}{2}|\br|.\nonumber
\end{eqnarray}
Remember that
$$|\bv|\leq c|\nabla p|.$$
We can infer from \eqref{vpe3} that there is a positive number $c$ such that
\begin{equation}
	\sup_{ t_0-\frac{1}{2}r^2\leq t\leq t_0+\frac{1}{2}r^2}\ibr(1+|\bv|)dx\leq c|\br|.\nonumber
\end{equation}
Let $\ell>\frac{N}{2}$ be given as in the lemma. 
We have
\begin{equation}
r^2	\ibr\qi dx\leq r^2\|\qi\|_{\ell, \br}|\br|^{1-\frac{1}{\ell}}\leq  cr^{2-\frac{N}{\ell}}|\br|.\nonumber
\end{equation}
Plugging the preceding three estimates into \eqref{vpe4} yields
\begin{eqnarray}
	\int_{B_{(1-\gamma)r}(x_0)}\varphi^2(u(x,t))dx&\leq&\frac{L^2\ln^22}{2}|\br|+\frac{cL}{\gamma ^2}|\br|\nonumber\\
	&&+cr^2L|\br|+\frac{cr^{2-\frac{N}{\ell}}L2^{1+L}}{\omega_r}|\br|.\label{vpe5}
\end{eqnarray}
Note that
$$\varphi^2(u)\geq \ln^2\left(\frac{\frac{\omega_r}{2}}{\frac{\omega_r}{2^L}}\right)=(L-1)^2\ln^22\ \ \mbox{on $B_{(1-\gamma)r}(x_0)\cap\left\{u(x,t)>M_r-\frac{\omega_r}{2^{1+L}}\right\}$.}$$
This combined with \eqref{vpe5} asserts
\begin{eqnarray}
\lefteqn{	\left|B_{(1-\gamma)r}(x_0)\cap\left\{u(x,t)>M_r-\frac{\omega_r}{2^{1+L}}\right\}\right|}\nonumber\\
&\leq& \frac{L^2}{2(L-1)^2}|\br|+\left(\frac{c}{L\gamma^2}+\frac{cr^2}{L}\right)|\br|+\frac{cr^{2-\frac{N}{\ell}}L2^{1+L}}{\omega_r}|\br|.\nonumber
\end{eqnarray}
Finally,
\begin{eqnarray}
\lefteqn{	\left|\br\cap\left\{u(x,t)>M_r-\frac{\omega_r}{2^{1+L}}\right\}\right|}\nonumber\\
&\leq&\left|B_{(1-\gamma)r}(x_0)\cap\left\{u(x,t)>M_r-\frac{\omega_r}{2^{1+L}}\right\}\right|+\left|\br\setminus B_{(1-\gamma)r}(x_0)\right|\nonumber\\
	&\leq&\left(\frac{L^2}{2(L-1)^2}+\frac{c}{L\gamma^2}+\frac{cr^2}{L}+\frac{cr^{2-\frac{N}{\ell}}L2^{1+L}}{\omega_r}+N\gamma\right)|\br|.\label{vpe6}
\end{eqnarray}
Now we take
$$\gamma=\frac{1}{8N}.$$
Then we pick $L$ so large that the sum of the first three terms between parentheses in \eqref{vpe6} satisfies
$$\frac{L^2}{2(L-1)^2}+\frac{c}{L\gamma^2}+\frac{cr^2}{L}\leq \frac{5}{8}.$$
The remaining term $\frac{cr^{2-\frac{N}{\ell}}L2^{1+L}}{\omega_r}$ there has two possibilities: Either 
$$\frac{cr^{2-\frac{N}{\ell}}L2^{1+L}}{\omega_r}\leq \frac{1}{8},\ \ \mbox{or}$$
$$\frac{cr^{2-\frac{N}{\ell}}L2^{1+L}}{\omega_r}>\frac{1}{8}.$$
In the former case , we have \eqref{po3}, while in the latter case we get \eqref{po4}. This completes the proof.
\end{proof}
	From here on, we assume \eqref{po3} holds. Then consider in $\qr$ the function
$$v=\ln\left(\frac{\omega_r+k(r)}{2^{s_1}(M_r-u)+k(r)}\right),$$
where
$$k(r)=r^{2-\frac{N}{s}}\sup_{ t_0-\frac{1}{2}r^2\leq t\leq t_0+\frac{1}{2}r^2}\|\qi\|_{s,\br}\ \ \mbox{and $s$ is given as in Lemma \ref{reh}}.$$
Since $s>2$ and we have assumed that $N=2, 3,$ or $4$, we always have
\begin{equation}\label{scon}
	2-\frac{N}{s}>0.
\end{equation} 
Thus, Lemma \ref{sob} holds for $\ell=s$.

We can easily verify that
\begin{eqnarray}
\lefteqn{\phx\pt v-\divg\left(\phx D(\bv)\nabla v\right)	+\phx D(\bv)\nabla v\cdot\nabla v+\nabla v\cdot\bv}\nonumber\\
&=&\frac{2^{s_1}\qi(\hat{u}-u)}{2^{s_1}(M_r-u)+k(r)}\leq \frac{2^{s_1+1}}{k(r)}\qi\ \ \mbox{in $\qr$.}\label{ve1}
\end{eqnarray}
The last step is due to (H3) and (D1).
\begin{lem}We have
	\begin{eqnarray}
		\sup_{Q_{\frac{r}{2}}(z_0)}v&\leq&c\left(\aiqr\left(v^+\right)^2dxdt\right)^{\frac{1}{2}}\nonumber\\
		&&+c\sup_{ t_0-\frac{1}{2}r^2\leq t\leq t_0+\frac{1}{2}r^2}\aibr v^+dx+c\sup_{ t_0-\frac{1}{2}r^2\leq t\leq t_0+\frac{1}{2}r^2}\left(\aibr|\bv|^sdx\right)^{\frac{1}{s}}+c.\label{hope7}	
	\end{eqnarray}
\end{lem}
\begin{proof}We employ a De Giorgi–Nash–Moser type of arguments. To this end,
 introduce the change of independent variables
$$\xi=\frac{x-x_0}{r}, \ \ \eta=\frac{t-t_0}{r^2}$$
and denote by $\tilde{f}(\xi, \eta)$ the composition $f(x_0+\xi r, t_0+r^2\eta)$ for a function $f$ defined on $\qr$. Subsequently, we can infer from \eqref{ve1} that
\begin{equation}\label{ve2}
	\phxt\pet \vt-\divg\left(\phxt D(\bvt)\nabla \vt\right)	+\phxt D(\bvt)\nabla \vt\cdot\nabla \vt+r\nabla \vt\cdot\bvt\leq \frac{2^{s_1+1}r^2}{k(r)}\qit\ \ \mbox{in $Q_1(0)$.}
\end{equation}
Pick a number $\delta\in(0,1)$. Define
$$r_n=1-\delta+\frac{\delta}{2^{n}},\ \ \ n=0,1,2,\cdots.$$
Choose a sequence of smooth cutoff functions $\zeta_n$ in $C^\infty_0(\mathbb{R}^{N+1})$ so that
\begin{eqnarray}
	\zeta_n(\xi,\eta)&=& 1\ \ \mbox{in $Q_{r_n}(0)$,}\nonumber\\
	\zeta_n(\xi,\eta)&=& 0\ \ \mbox{outside $Q_{r_{n-1}}(0)$,}\nonumber\\
	|\nabla\zeta_n(\xi,\eta)|&\leq&\frac{c2^n}{\delta}, \ \ |\pet\zeta_n(\xi,\eta)|\leq\frac{c2^n}{(1-\delta)\delta}\ \ \mbox{on $Q_{r_{n-1}}(0)$, and}\nonumber\\
	0&\leq &\zeta_n(\xi,\eta)\leq 1\ \ \mbox{on $Q_{r_{n-1}}(0)$. }\nonumber
\end{eqnarray}
Select
\begin{equation}\label{kcon}
	k\geq 2\left(\sup_{-\frac{1}{2}\leq\eta\leq \frac{1}{2}}\|\vt^+\|_{s,B_1(0)}+\sup_{-\frac{1}{2}\leq\eta\leq \frac{1}{2}}\|\bvt\|_{s,B_1(0)}+1\right)
\end{equation}
as below. Set
$$k_n=k-\frac{k}{2^{n+1}},\ \ n=0,1,2,\cdots.$$
Use $(\vt-k_{n+1})^+\zeta_{n+1}^2$ as a test function in \eqref{ve2} to get
\begin{eqnarray}
	\lefteqn{\frac{1}{2}\frac{d}{d\eta}\ibn\phxt\left[(\vt-k_{n+1})^+\right]^2\zeta_{n+1}^2d\xi+\ibn\phxt D(\bvt)\nabla\vt\cdot\nabla(\vt-k_{n+1})^+\zeta_{n+1}^2d\xi}\nonumber\\
	&&+\ibn\phxt D(\bvt)\nabla\vt\cdot\nabla\vt(\vt-k_{n+1})^+\zeta_{n+1}^2d\xi\nonumber\\
	&\leq&\ibn\phxt\left[(\vt-k_{n+1})^+\right]^2\zeta_{n+1}\pet\zeta_{n+1}d\xi\nonumber\\
	&&-2\ibn\phxt D(\bvt)\nabla\vt\cdot\nabla \zeta_{n+1}(\vt-k_{n+1})^+\zeta_{n+1}d\xi\nonumber\\
	&&-r\ibn\nabla\vt\cdot\bvt(\vt-k_{n+1})^+\zeta_{n+1}^2d\xi+\frac{cr^2}{k(r)}\ibn\qit(\vt-k_{n+1})^+\zeta_{n+1}^2d\xi.\label{hope1}
\end{eqnarray}
Note that
$$\nabla\vt=\nabla(\vt-k_{n+1})^+\ \ \mbox{in $\{\vt\geq k_{n+1}\}$}.$$
Arguing the same way as in the proof of \eqref{vpe4}, we arrive at
\begin{eqnarray}
	a_n+b_n&\equiv&\sup_{-r_n^2\leq\eta\leq 0}\ibn\left[(\vt-k_{n+1})^+\right]^2\zeta_{n+1}^2d\xi+\int_{\qrn}\left|\nabla\left[(\vt-k_{n+1})^+\zeta_{n+1}\right]\right|^2d\xi d\eta\nonumber\\
	&\leq&\frac{c4^n}{\delta^2(1-\delta)}\int_{\qrn}\left[(\vt-k_{n+1})^+\right]^2d\xi d\eta+\frac{c4^n}{\delta^2}\int_{\qrn}(1+|\bvt|)(\vt-k_{n+1})^+d\xi d\eta\nonumber\\
	&&+cr^2\int_{\qrn}|\bvt|(\vt-k_{n+1})^+d\xi d\eta+\frac{cr^2}{k(r)}\int_{\qrn}\qit(\vt-k_{n+1})^+d\xi d\eta.\label{hope2}
\end{eqnarray}
Introduce the sequence
$$y_n=\int_{-\frac{1}{2}r_n^2}^{\frac{1}{2}r_n^2}\|(\vt-k_{n+1})^+\|_{\frac{s}{s-1},B_{r_n}(0)}d\eta.$$
We wish to apply Lemma 4.1 in (\cite{D}, p.12) to this sequence. To this end, we estimate, with the aid of  \eqref{kcon}, that
\begin{eqnarray}
	\int_{\qrn}\left[(\vt-k_{n+1})^+\right]^2d\xi d\eta&\leq&\int_{-\frac{1}{2}r_n^2}^{\frac{1}{2}r_n^2}\|(\vt-k_{n+1})^+\|_{s,B_{r_n}(0)}|(\vt-k_{n+1})^+\|_{\frac{s}{s-1},B_{r_n}(0)}d\eta\nonumber\\
	&\leq&\sup_{-\frac{1}{2}r_n^2\leq\eta\leq \frac{1}{2}r_n^2}\|(\vt-k_{n+1})^+\|_{s,B_{r_n}(0)}y_n\leq \frac{k}{2}y_n.\nonumber
\end{eqnarray}
 Similarly,
\begin{eqnarray}
	\int_{\qrn}(1+|\bvt|)(\vt-k_{n+1})^+d\xi d\eta&\leq&\int_{-\frac{1}{2}r_n^2}^{\frac{1}{2}r_n^2}\|1+|\bvt|\|_{s,B_{r_n}(0)}|(\vt-k_{n+1})^+\|_{\frac{s}{s-1},B_{r_n}(0)}d\eta\nonumber\\
	&\leq&\sup_{-\frac{1}{2}r_n^2\leq\eta\leq \frac{1}{2}r_n^2}\|1+|\bvt|\|_{s,B_{r_n}(0)}y_n\leq cky_n,\nonumber\\
	\frac{cr^2}{k(r)}\int_{\qrn}\qit(\vt-k_{n+1})^+d\xi d\eta&\leq&\frac{cr^2}{k(r)}\sup_{-\frac{1}{2}\leq\eta\leq \frac{1}{2}}\|\qit\|_{s,B_{1}(0)}y_n\nonumber\\
	&\leq&\frac{cr^{2-\frac{N}{s}}}{k(r)}\sup_{t_0-\frac{1}{2}r^2\leq\eta\leq t_0+\frac{1}{2}r^2}\|\qit\|_{s,B_{r}(x_0)}y_n\leq cy_n.\nonumber
\end{eqnarray}
Here the last step is due to the definition of $k(r)$. Collecting these estimates in \eqref{hope2} yields
\begin{equation}\label{hope3}
	a_n+b_n\leq \frac{c4^nk}{\delta^{2}(1-\delta)}y_n.
\end{equation}
Set
$$A_n(\eta)=B_{r_n}(0)\cap\left\{\vt(\xi,\eta)\geq k_{n+1}\right\}.$$
Let $\ve\in (0,s)$ be selected as below. We deduce from the Sobolev embedding theorem that
\begin{eqnarray}
	y_{n+1}&\leq&\int_{-\frac{1}{2}r_n^2}^{\frac{1}{2}r_n^2}\left(\ibn\left[(\vt-k_{n+1})^+\zeta_{n+1}\right]^{\frac{s}{s-1}}d\xi\right)^{\frac{s-1}{s}}d\eta\nonumber\\
	&\leq&\int_{-\frac{1}{2}r_n^2}^{\frac{1}{2}r_n^2}\left(\ibn\left[(\vt-k_{n+1})^+\zeta_{n+1}\right]^{\frac{s-\ve}{s-1}+\frac{\ve}{s-1}}d\xi\right)^{\frac{s-1}{s}}d\eta\nonumber\\
	&\leq&\int_{-\frac{1}{2}r_n^2}^{\frac{1}{2}r_n^2}\|(\vt-k_{n+1})^+\zeta_{n+1}\|_{\frac{2N}{N-2}, B_{r_n}(0)}^{\frac{s-\ve}{s}}\|(\vt-k_{n+1})^+\zeta_{n+1}\|_{\frac{2N\ve}{2N(s-1)-(N-2)(s-\ve)}, B_{r_n}(0)}^{\frac{\ve}{s}}d\eta\nonumber\\
		&\leq&ca_n^{\frac{\ve}{2s}}\int_{-\frac{1}{2}r_n^2}^{\frac{1}{2}r_n^2}\|\nabla\left[(\vt-k_{n+1})^+\zeta_{n+1}\right]\|_{2, B_{r_n}(0)}^{\frac{s-\ve}{s}}|A_n(\eta)|^{\frac{2N(s-1)-(N-2)(s-\ve)-N\ve}{2Ns}}d\eta\nonumber\\
			&\leq&ca_n^{\frac{\ve}{2s}}b_n^{\frac{s-\ve}{2s}}\left(\int_{-\frac{1}{2}r_n^2}^{\frac{1}{2}r_n^2}|A_n(\eta)|^{\frac{2N(s-1)-(N-2)(s-\ve)-N\ve}{N(s+\ve)}}d\eta\right)^{\frac{s+\ve}{2s}}.\label{hope4}
\end{eqnarray}
Observe that
\begin{equation}\label{hope5}
	y_n\geq \int_{-\frac{1}{2}r_n^2}^{\frac{1}{2}r_n^2}\left(\int_{A_n(\eta)}\left[(\vt-k_{n})^+\right]^{\frac{s}{s-1}}d\xi\right)^{\frac{s-1}{s}}d\eta\geq\frac{k}{2^{n+2}}\int_{-\frac{1}{2}r_n^2}^{\frac{1}{2}r_n^2}|A_n(\eta)|^{\frac{s}{s-1}}d\eta.
\end{equation}
Thus, we choose $\ve$ so that 
$$\frac{2N(s-1)-(N-2)(s-\ve)-N\ve}{N(s+\ve)}=\frac{s}{s-1}.$$
Solve it to obtain
$$\ve=\frac{s(2s-N)}{sN+2s-N}.$$
According to \eqref{scon}, the number $\ve$ so obtained lies in $(0,s)$. Keep this $\ve$. Then use  \eqref{hope5}, and \eqref{hope3} in \eqref{hope4} to obtain
\begin{eqnarray}
	y_{n+1}&\leq &\left(\frac{c4^nk}{\delta^{2}(1-\delta)}y_n\right)^{\frac{1}{2}}\left(\frac{2^{n+2}y_n}{k}\right)^{\frac{s+\ve}{2s}}\nonumber\\
	&\leq&\frac{c4^n}{\delta\sqrt{1-\delta}k^{\frac{\ve}{2s}}}y_n^{1+\frac{\ve}{2s}}.\nonumber
\end{eqnarray}
Lemma 4.1 in (\cite{D}, p.12) asserts if we choose $k$ so large that
$$y_0\leq \left(\frac{\delta\sqrt{1-\delta}k^{\frac{\ve}{2s}}}{c}\right)^{\frac{2s}{\ve}}\left(\frac{1}{4}\right)^{\frac{4s^2}{\ve^2}},$$
then
\begin{equation}
	\sup_{Q_{1-\delta}(0)}\vt\leq k.\nonumber
\end{equation}
In view of \eqref{kcon}, it is enough for us to take
$$k=\frac{c}{\delta^{\frac{2s}{\ve}}(1-\delta)^{\frac{s}{\ve}}}y_0+2\left(\sup_{-\frac{1}{2}\leq\eta\leq \frac{1}{2}}\|\vt^+\|_{s,B_1(0)}+\sup_{-\frac{1}{2}\leq\eta\leq \frac{1}{2}}\|\bvt\|_{s,B_1(0)}+1\right).$$
Recall that
$$y_0=\int_{-\frac{1}{2}}^{\frac{1}{2}}\left\|\left(\vt-\frac{k}{2}\right)^+\right\|_{\frac{s}{s-1}, B_1(0)}d\eta\leq |B_1(0)|^{\frac{s-1}{s}-\frac{1}{2}}\|\vt^+\|_{2,Q_1(0)}.$$
Therefore, we have
\begin{eqnarray}
	\sup_{Q_{1-\delta}(0)}\vt&\leq&2\sup_{-\frac{1}{2}\leq\eta\leq \frac{1}{2}}\|\vt^+\|_{s,B_1(0)}+\frac{c}{\delta^{\frac{2s}{\ve}}(1-\delta)^{\frac{s}{\ve}}}\|\vt^+\|_{2,Q_1(0)}\nonumber\\
	&&+2\sup_{-\frac{1}{2}\leq\eta\leq \frac{1}{2}}\|\bvt\|_{s,B_1(0)}+2.\label{hope10}
\end{eqnarray}
According to the interpolation inequality (\cite{GT}, p.146), we have
\begin{equation}
	\|\vt^+\|_{s,B_1(0)}\leq \sigma\|\vt^+\|_{\infty,B_1(0)}+\frac{1}{\sigma^{s-1}}\|\vt^+\|_{1,B_1(0)},\ \ \sigma>0.\nonumber
\end{equation}
Use this in \eqref{hope10} appropriately and keep in mind the fact that we may assume $ \sup_{Q_{1-\delta}(0)}\vt^+=\sup_{Q_{1-\delta}(0)}\vt$ to derive
\begin{eqnarray}
\sup_{Q_{1-\delta}(0)}\vt&\leq&\sigma\sup_{Q_{1}(0)}\vt+\frac{2^s}{\sigma^{s-1}}\sup_{-\frac{1}{2}\leq\eta\leq \frac{1}{2}}\|\vt^+\|_{1,B_1(0)}\nonumber\\
&&+\frac{c}{\delta^{\frac{2s}{\ve}}(1-\delta)^{\frac{s}{\ve}}}\|\vt^+\|_{2,Q_1(0)}+2\sup_{-\frac{1}{2}\leq\eta\leq \frac{1}{2}}\|\bvt\|_{s,B_1(0)}+2.\label{hope9}	\nonumber
\end{eqnarray}
Return to the $(x, t)$ variables to get
\begin{eqnarray}
	\sup_{Q_{(1-\delta)r}(z_0)}v&\leq&\sigma\sup_{Q_{r}(z_0)}v+\frac{2^s}{\sigma^{s-1}}\sup_{ t_0-\frac{1}{2}r^2\leq t\leq t_0+\frac{1}{2}r^2}\aibr v^+dx\nonumber\\
	&&+\frac{c}{\delta^{\frac{2s}{\ve}}(1-\delta)^{\frac{s}{\ve}}}\left(\aiqr\left(v^+\right)^2dxdt\right)^{\frac{1}{2}}\nonumber\\
	&&+2\sup_{ t_0-\frac{1}{2}r^2\leq t\leq t_0+\frac{1}{2}r^2}\left(\aibr|\bv|^sdx\right)^{\frac{1}{s}}+2.\nonumber	
\end{eqnarray}
Take $\delta=\frac{1}{2^{n+2}-1}$. Then replace $r$ by $r_{n+1}\equiv r-\frac{r}{2^{n+2}}$ in the above inequality to derive
\begin{eqnarray}
	\sup_{Q_{r_n}(z_0)}v&\leq&\sigma\sup_{Q_{r_{n+1}}(z_0)}v+\frac{c}{\sigma^{s-1}}\sup_{ t_0-\frac{1}{2}r^2\leq t\leq t_0+\frac{1}{2}r^2}\aibr v^+dx\nonumber\\
	&&+c4^{\frac{sn}{\ve}}\left(\aiqr\left(v^+\right)^2dxdt\right)^{\frac{1}{2}}+c\sup_{ t_0-\frac{1}{2}r^2\leq t\leq t_0+\frac{1}{2}r^2}\left(\aibr|\bv|^sdx\right)^{\frac{1}{s}}+2.	\nonumber
\end{eqnarray}
Iterate over $n$ as in (\cite{D}, p.13) to deduce
\begin{eqnarray}
\lefteqn{	\sup_{Q_{\frac{r}{2}}(z_0)}v}\nonumber\\
&\leq&\sigma^n\sup_{Q_{r_{n}}(z_0)}v+c\left(\aiqr\left(v^+\right)^2dxdt\right)^{\frac{1}{2}}\sum_{i=0}^{n-1}\left(4^{\frac{s}{\ve}}\sigma\right)^i\nonumber\\
	&&+\left(\frac{c}{\sigma^{s-1}}\sup_{ t_0-\frac{1}{2}r^2\leq t\leq t_0+\frac{1}{2}r^2}\aibr v^+dx+c\sup_{ t_0-\frac{1}{2}r^2\leq t\leq t_0+\frac{1}{2}r^2}\left(\aibr|\bv|^sdx\right)^{\frac{1}{s}}+2\right)\sum_{i=0}^{n-1}\sigma^i.\label{hope8}
\end{eqnarray}
Choose $\sigma$ suitably small. Then take $n\ra \infty$ to arrive at \eqref{hope7}. The proof is complete.
\end{proof}
\begin{lem}\label{lab}We have
	$$\limsup_{\rho\ra 0}\left(\dashint_{Q_\rho(z_0)}\left(v^+\right)^2dxdt+\sup_{t_0-\rho^2\leq t\leq t_0}\dashint_{B_\rho(x_0)} v^+dx\right)<\infty.$$
\end{lem}
\begin{proof}Note from the definition of $v$ that
	$$v(x,t) >0 \Longleftrightarrow u(x,t)>M_r-\frac{\omega_r}{2^{s_1}}.$$
	This together with \eqref{po3} implies
\begin{equation}\label{hope11}
		|\{x\in\br, v(x,t)=0\}|\geq \frac{1}{8}|\br|\ \ \mbox{for each $t\in\itz$}.
\end{equation}
	Pick a number $\delta\in(0,1)$. 	Let $\theta(\eta)$ be a smooth decreasing function on $[0,r]$ with the properties
	$$\theta(\eta)=1 \ \ \mbox{for $\eta\leq (1-\delta)r$, $\theta(r)=0$, and $|\theta^\prime(\eta)|\leq \frac{c}{r\delta}$.}$$
	We easily see that
	$$\left|\br\setminus B_{(1-\delta)r}(x_0)\right|\leq N\delta|\br|.$$
It is enough for us to take
\begin{equation}\label{hope15}
	 N\delta=\frac{1}{16}. 
\end{equation}
Combining this with  \eqref{hope11} yields
			\begin{equation}
			|\{\theta(|x-x_0|)=1\}\cap\{x\in B_{r}, v(x,t)=0\}|\geq \frac{1}{16}|\br|\ \ \mbox{for each $t\in\itz$}	.	
		\end{equation}
	This puts us in a position to apply Proposition 2.1 in (\cite{D}, p.5). Upon doing so, we obtain
	\begin{equation}\label{hope12}
		\ibr \theta^2(|x-x_0|)v^2dx\leq cr^2\ibr \theta^2(|x-x_0|)|\nabla v|^2dx.
	\end{equation}

Next, we define
	$$f(t)=\left\{\begin{array}{ll}
		1&\mbox{if $t\in (t_0-(1-\delta)^2r^2, t_0+\frac{1}{2}r^2]$,}\\
		\frac{1}{\left[\frac{1}{2}-(1-\delta)^2\right]r^2}(t-t_0+\frac{1}{2}r^2)&\mbox{if $t\in [t_0-\frac{1}{2}r^2,t_0-(1-\delta)^2r^2 ]$,}
	\end{array}\right.$$
where $\delta$ is given as in \eqref{hope15}.
	We use $\theta^2(|x-x_0|)f^2(t)$ as a test function in \eqref{ve1} to derive
	\begin{eqnarray}
	\lefteqn{	\ibr\phx v\theta^2(|x-x_0|)f^2(t)dx+\iqr\phx D(\bv)\nabla v\cdot\nabla v\theta^2(|x-x_0|)f^2(t)dxdt}\nonumber\\
	&\leq&-2\iqr\phx D(\bv)\nabla v\theta(|x-x_0|)\nabla\theta(|x-x_0|)f^2(t)dxdt\nonumber\\
	&&+2\iqr\phx v\theta^2(|x-x_0|)f(t)f^\prime(t)dxdt+\iqr\nabla v\cdot\bv\theta^2(|x-x_0|)f^2(t)dxdt\nonumber\\ &&+\frac{c}{k(r)}\iqr\qi\theta^2(|x-x_0|)f^2(t)dxdt.\nonumber
	\end{eqnarray}
As before, we can derive from \eqref{ddef} and \eqref{ellip} that
	\begin{eqnarray}
	\lefteqn{\sup_{ t_0-\frac{1}{2}r^2\leq t\leq t_0+\frac{1}{2}r^2}	\ibr v\theta^2(|x-x_0|)f^2(t)dx+\iqr(1+|\bv|)|\nabla v|^2\theta^2(|x-x_0|)f^2(t)dxdt}\nonumber\\
	&\leq&\frac{c}{r^2}\iqr(1+|\bv|)f^2(t)dxdt+\frac{c}{r^2}\iqr v\theta^2(|x-x_0|)f(t)dxdt\nonumber\\
	&&+c\iqr|\bv|\theta^2(|x-x_0|)f^2(t)dxdt+\frac{c}{k(r)}\iqr\qi\theta^2(|x-x_0|)f^2(t)dxdt.\label{hope13}
\end{eqnarray}
It follows from \eqref{vpe3} that
\begin{eqnarray}
	\iqr(1+|\bv|)f^2(t)dxdt\leq cr^2\sup_{ t_0-\frac{1}{2}r^2\leq t\leq t_0+\frac{1}{2}r^2}\ibr(1+|\nabla p|)dx\leq cr^{N+2}.\nonumber
\end{eqnarray}
For each $\ve>0$ we deduce from \eqref{hope12} that
\begin{eqnarray}
	\frac{c}{r^2}\iqr v\theta^2(|x-x_0|)f(t)dxdt&\leq &\frac{\ve}{r^2}\iqr v^2\theta^2(|x-x_0|)f^2(t)dxdt+\frac{cr^N}{\ve}\nonumber\\
	&\leq &c\ve\iqr |\nabla v|^2\theta^2(|x-x_0|)f^2(t)dxdt+\frac{cr^N}{\ve}.\nonumber
\end{eqnarray}
The last term in \eqref{hope13} can be estimated as follows
\begin{eqnarray}
	\frac{c}{k(r)}\iqr\qi\theta^2(|x-x_0|)f^2(t)dxdt&\leq&\frac{cr^2}{k(r)}\sup_{ t_0-\frac{1}{2}r^2\leq t\leq t_0+\frac{1}{2}r^2}\|\qi\|_{s,\br}|\br|^{1-\frac{1}{s}}\nonumber\\
	&\leq&cr^N.\nonumber
\end{eqnarray}
Plug the preceding estimates into \eqref{hope13}, choose $\ve$ suitably small in the resulting inequality, and thereby obtain
\begin{equation}
	\sup_{ t_0-\frac{1}{2}r^2\leq t\leq t_0+\frac{1}{2}r^2}	\aibr v\theta^2(|x-x_0|)f^2(t)dx+\frac{1}{r^N}\iqr(1+|\bv|)|\nabla v|^2\theta^2(|x-x_0|)f^2(t)dxdt\leq c.\nonumber
\end{equation}
Apply \eqref{hope12} again to derive
\begin{equation}
\sup_{t_0-(1-\delta)^2r^2\leq t\leq t_0+\frac{1}{2}r^2}	\dashint_{B_{(1-\delta)r}(x_0)} v	dx+\dashint_{Q_{(1-\delta)r}(z_0)} v^2dxdt\leq c.\nonumber
\end{equation}
This implies the desired result.
\end{proof}
We can conclude from \eqref{haha2} and \eqref{vpe3} that
$$\limsup_{r\ra 0}\sup_{ t_0-\frac{1}{2}r^2\leq t\leq t_0+\frac{1}{2}r^2}\left(\aibr|\bv|^sdx\right)^{\frac{1}{s}}\leq c.$$
Combining this with \eqref{hope7} and Lemma \ref{lab} yields
$$\sup_{\qrh}v\leq c.$$
Recall from the definition of $v$ to deduce
\begin{eqnarray}
	\omega_{\frac{r}{2}}\leq \frac{e^c2^{s_1}-1}{e^c2^{s_1}}\omega_r+\frac{e^c-1}{e^c2^{s_1}}k(r).\label{hope14}
\end{eqnarray}
Note that $\omega_r$ is an increasing function of $r$. We can conclude from \eqref{po4}, \eqref{hope14}, and our assumptions on $\qi$ that there exist $\gamma\in(0,1)$ and $c>0$ such that
$$\omega_{\frac{r}{2}}\leq \gamma\omega_r+cr^{2-\frac{N}{s}}\ \ \mbox{for each $r>0$ with $Q_{2r}(z_0)\subset \ot$.}$$
This enables us to invoke Lemma 8.23 in (\cite{GT}, p.201) to get
\begin{equation}\label{hope16}
	\omega_r\leq cr^\alpha \ \ \mbox{for some $c>0, \alpha>0$ and $r$ sufficiently small.}
\end{equation}
We are ready to prove our partial regularity result.
\begin{lem}Let $z_0\in\ot, r>0$ be such that \eqref{hope16} holds. Then $z_0$ is a regular point.
	\end{lem}
\begin{proof} Let $z_0\in\ot, r>0$ be given as in the lemma.
	For each $\rho\in (0,r)$ we select a cutoff function $\zeta\in C^\infty_0(\br)$ with the properties
$$\zeta=1 \ \ \mbox{on $B_\rho(x_0)$, $0\leq\zeta\leq 1$ on $\br$, and $|\nabla\zeta|\leq \frac{c}{r-\rho}$.}$$
Set
$$w=p\zeta.$$
We easily verify that $w$ satisfies the equation
\begin{equation}\label{hop1}
	-\divg\left(\epc\nabla w\right)=(\qi-\qp)\zeta-\epc\nabla p\cdot\nabla\zeta-\divg\left(\epc\nabla\zeta p\right)\ \ \mbox{in $\rn$.}
\end{equation}
The time variable $t$ in the equation is assumed to lie in $\itz$. We can say that the above equation holds on $\rn$ because the support of $w$ lies inside $\br$.
This fact will be used in some subsequent calculations without acknowledgment. We easily see from \eqref{AA}
that
\begin{equation}\label{AA1}
\|w\|_{\infty,\br}\leq \|p\|_{\infty,\br}\leq c.
\end{equation}
For each $y\in\rn$ and $R>0$, we consider the boundary value problem
\begin{eqnarray}
	-\divg\left(\frac{1}{\mu(u(x_0,t_0))}(K)_{x_0,r}\nabla\phi\right) &=& 0 \ \ \mbox{in $\bry$},\label{hop2}\\
	\phi&=&w\ \ \mbox{on $\partial\bry$.}
\end{eqnarray}
Here 
$$u(x_0,t_0)=\lim_{r\ra 0}M_r=\lim_{r\ra 0}m_r.$$ Obviously,  the matrix $(K)_{x_0,r}$ satisfies the ellipticity condition \eqref{add5} uniformly for small $r$.  We can appeal to a result in (\cite{G}, p.78) to obtain
\begin{equation}\label{hop6}
	\int_{\bdy}\left|\nabla\phi-\left(\nabla\phi\right)_{y,\delta}\right|^2dx\leq c\left(\frac{\delta}{R}\right)^{N+2}	\int_{\bry}\left|\nabla\phi-\left(\nabla\phi\right)_{y,R}\right|^2dx\ \ \mbox{for $0<\delta\leq R$}.
\end{equation}
Subtract \eqref{hop1} from \eqref{hop2} and use $\phi-w$ as a test function in the resulting equation to derive
\begin{eqnarray}
	\lefteqn{\frac{1}{\mu(u(x_0,t_0))}\int_{\bry}(K)_{x_0,r}\nabla\phi\cdot\nabla(\phi-w)-\int_{\bry}\epc\nabla w\cdot\nabla(\phi-w)dx}\nonumber\\
	&=&-\int_{\bry}(\qi-\qp)\zeta(\phi-w)dx+\int_{\bry}\epc\nabla p\cdot\nabla\zeta(\phi-w)dx\nonumber\\
	&&-\int_{\bry}\epc\nabla\zeta p\cdot\nabla(\phi-w)dx.\label{hop3}
\end{eqnarray}
It follows from the maximum principle and \eqref{AA1} that
$$\|\phi\|_{\infty,\bry}\leq \|w\|_{\infty,\bry}\leq c.$$
Equipped with these, we derive from \eqref{hop3} that
\begin{eqnarray}
	\int_{\bry}|\nabla(\phi-w)|^2dx&\leq&	c\int_{\bry}\left|\left(\frac{1}{\mu(u(x_0,t_0))}(K)_{x_0,r}-\epc\right)\nabla w\right|^2dx\nonumber\\
	&&+c\int_{\bry}|\qi-\qp|\chi_{\br} dx+c\int_{\bry}|\nabla p\cdot\nabla\zeta|dx\nonumber\\
	&&+c\int_{\bry}|\nabla\zeta|^2dx.
	\label{hop7}
\end{eqnarray}
Set
\begin{equation}\label{etaz}
\eta(r)=\sup_{Q_r(z_0)}\left|\mu(u(x_0,t_0))-\mu(u)\right|+\left(\textup{osc}_{\br}K(x)\right)^2.	
\end{equation}
In view of \eqref{hope16}, (H5), and \eqref{in3}, we have
\begin{equation}\label{eta2}
	\lim_{r\ra 0^+}\eta(r)=0.
\end{equation}
Recall that $w=0$ outside $\br$. Keeping this and \eqref{etaz} in mind, we estimate the second integral in \eqref{hop7} to obtain  that
\begin{eqnarray}
	\lefteqn{\int_{\bry}\left|\left(\frac{1}{\mu(u(x_0,t_0))}(K)_{x_0,r}-\epc\right)\nabla w\right|^2dx}\nonumber\\
	&\leq&c\int_{\bry}\left|\left((K)_{x_0,r}-K(x)\right)\nabla w\right|^2dx+c\int_{\bry}\left|\mu(u(x_0,t_0))-\mu(u)\right|\left|\nabla w\right|^2dx\nonumber\\
		&\leq&c\eta(r)\int_{\bry}\left|\nabla w\right|^2dx.\label{hop5}\nonumber
\end{eqnarray}
Combining this with \eqref{hop7} yields that
\begin{eqnarray}
	\int_{\bry}|\nabla(\phi-w)|^2dx&\leq&c\eta(r)\int_{\bry}\left|\nabla w\right|^2dx+cR^NE(y),\nonumber
\end{eqnarray}
where 
$$E(y)=M((\qi-\qp)\chi_{\br})(y)+M(\nabla p\cdot\nabla\zeta)(y)+M(|\nabla\xi|^2)(y).$$
It is easy to verify
$$\ibry|f-(f)_{y,R}|^2dx\leq \ibry|f-c|^2dx\ \ \mbox{for each number $c\in \mathbb{R}$ and $f\in L^1_{\textup{loc}}(\mathbb{R}^N)$}.$$
Consequently,
\begin{eqnarray}
	\int_{\bdy}\left|\nabla (w-\phi)-\left(\nabla (w-\phi)\right)_{y,\delta}\right|^2dx&\leq&\int_{\bdy}\left|\nabla (w-\phi)-\left(\nabla (w-\phi)\right)_{y,R}\right|^2dx\nonumber\\
	&\leq&\int_{\bry}\left|\nabla (w-\phi)-\left(\nabla (w-\phi)\right)_{y,R}\right|^2dx.\nonumber
\end{eqnarray}
With this in mind, we derive from \eqref{hop6} that
\begin{eqnarray}
	\lefteqn{\int_{\bdy}\left|\nabla w-\left(\nabla w\right)_{y,\delta}\right|^2dx}\nonumber\\
	&\leq&2\int_{\bdy}\left|\nabla \phi-\left(\nabla \phi\right)_{y,\delta}\right|^2dx+2\int_{\bdy}\left|\nabla (w-\phi)-\left(\nabla (w-\phi)\right)_{y,\delta}\right|^2dx\nonumber\\
	&\leq&c\left(\frac{\delta}{R}\right)^{N+2}	\int_{\bry}\left|\nabla\phi-\left(\nabla\phi\right)_{y,R}\right|^2dx+c\int_{\bry}\left|\nabla (w-\phi)\right|^2dx\nonumber\\
	&\leq&c\left(\frac{\delta}{R}\right)^{N+2}\int_{\bry}\left|\nabla w-\left(\nabla w\right)_{y,R}\right|^2dx+c\int_{\bry}\left|\nabla (w-\phi)\right|^2dx\nonumber\\
		&\leq&c\left(\left(\frac{\delta}{R}\right)^{N+2}+\eta(r)\right)\int_{\bry}\left|\nabla w\right|^2dx+cR^NE(y).\label{hop8}
\end{eqnarray}
Let $\tau\in (0,1)$ to be determined. Take $\delta=\tau R$ in \eqref{hop8} to get
\begin{eqnarray}
	\dashint_{B_{\tau R}(y)}\left|\nabla w-\left(\nabla w\right)_{y,\tau R}\right|^2dx\leq c\left(\tau^2+\frac{\eta(r)}{\tau^N}\right)M(|\nabla w|^2)(y)+\frac{c}{\tau^N}E(y).\nonumber
\end{eqnarray}
Since this is true for each $R>0$, we obtain
\begin{equation}\label{hop9}
	\left(\left(\nabla w\right)^\#\right)^2\leq c\left(\tau^2+\frac{\eta(r)}{\tau^N}\right)M(|\nabla w|^2)+\frac{c}{\tau^N}E.
\end{equation}
In view of the construction of our weak solution in Section 2, we may assume 
$$|\nabla p|\in L^\ell(\Omega)\ \  \mbox{for each $\ell>1$.}$$
Now fix $\ell>2$.  We derive from the Fefferman-Stein inequality \eqref{fs}, \eqref{hop9}, and the Hardy-Littlewood maximal theorem \eqref{hl} that
\begin{eqnarray}
	\int_{\rn}|\nabla w|^\ell dx
	&\leq&c\int_{\rn}|\left(\nabla w\right)^\#|^\ell dx\nonumber\\
	&\leq&	c\left(\tau^2+\frac{\eta(r)}{\tau^N}\right)^{\frac{\ell}{2}}\int_{\rn}M^{\frac{\ell}{2}}(|\nabla w|^2)dx+\frac{c}{\tau^{\frac{N\ell}{2}}}\int_{\rn}E^{\frac{\ell}{2}}dx\nonumber\\
		&\leq&	c\left(\tau^2+\frac{\eta(r)}{\tau^N}\right)^{\frac{\ell}{2}}	\int_{\rn}|\nabla w|^\ell dx+\frac{c}{\tau^{\frac{N\ell}{2}}}\int_{\rn}|(\qi-\qp)\chi_{\br}|^{\frac{\ell}{2}}dx\nonumber\\
		&&+\frac{c}{\tau^{\frac{N\ell}{2}}}\int_{\rn}|\nabla p\cdot\nabla\zeta|^{\frac{\ell}{2}}dx+\frac{c}{\tau^{\frac{N\ell}{2}}}\int_{\rn}|\nabla\zeta|^{\ell}dx.\label{hop10}
\end{eqnarray}
We can choose $\tau\in (0,1), r\in(0,1)$ such that the coefficient of the first integral on the right-hand side satisfies
$$	c\left(\tau^2+\frac{\eta(r)}{\tau^N}\right)^{\frac{\ell}{2}}\leq\frac{1}{2}.$$
This is possible because
$$\lim_{\tau\ra 0^+}\lim_{r\ra 0^+}\left(\tau^2+\frac{\eta(r)}{\tau^N}\right)^{\frac{\ell}{2}}=0\ \ \mbox{due to \eqref{eta2}}.$$
Let $\tau, r$ so chosen. We can deduce from \eqref{hop10} that
\begin{eqnarray}
	\int_{B_{\rho}(x_0)}|\nabla p|^\ell dx&\leq& \frac{c}{(r-\rho)^{\frac{\ell}{2}}}\ibr|\nabla p|^{\frac{\ell}{2}}dx+\frac{cr^N}{(r-\rho)^\ell}+c\ibr|\qi-\qp|^{\frac{\ell}{2}}dx\nonumber\\
	&\leq&\ve\ibr|\nabla p|^{\ell}dx+\frac{cr^N}{\ve(r-\rho)^\ell}+c\ibr|\qi-\qp|^{\frac{\ell}{2}}dx,\ \ \ve>0.\nonumber
\end{eqnarray}
By the proof of \eqref{hope8}, we obtain
\begin{eqnarray}
		\int_{B_{\frac{r}{2}}(x_0)}|\nabla p|^\ell dx\leq \frac{c}{r^{-N+\ell}}+c\ibr|\qi-\qp|^{\frac{\ell}{2}}dx.\nonumber
\end{eqnarray}
Since this is true for a.e. $t\in\itz$, we arrive at
\begin{equation}
\sup_{\itz}\int_{B_{\frac{r}{2}}(x_0)}|\nabla p|^\ell dx\leq \frac{c}{r^{-N+\ell}}+c\sup_{ t_0-\frac{1}{2}r^2\leq t\leq t_0+\frac{1}{2}r^2}\ibr|\qi-\qp|^{\frac{\ell}{2}}dx.\nonumber
\end{equation}
The proof is complete.
\end{proof}

\end{document}